\documentclass[a4paper]{article}
\usepackage[latin1]{inputenc}
\usepackage{amsthm, amssymb,amscd,amsmath,amsfonts}
\usepackage{hyperref}
\hypersetup{colorlinks=true, citecolor=red, urlcolor=blue,linkcolor=blue}
\usepackage{color}

\newtheorem{theorem}{Theorem}
\newtheorem{proposition}{Proposition}
\newtheorem{lemma}{Lemma}
\newtheorem{remark}{Remark}
\numberwithin{equation}{section}

\newcommand{\<}{\left\langle}
\renewcommand{\>}{\right\rangle}

\def\N{\mathbb N}
\def\R{\mathbb R}
\def\Pb{\mathbb P}
\def\E{\mathbb{E}\,}

\newcommand{\Const}{\mbox{\rm Const}}
\newcommand{\Cov}{\mbox{\rm Cov}}
\newcommand{\dist}{{\rm dist}}
\newcommand{\diag}{\mbox{\rm diag}}
\newcommand{\Vol}{\mbox{\rm Vol}}
\newcommand{\Var}{\mbox{\rm Var}}
\def\indicator{\mathbf{1}}
\def\bP{{\mathbf P}_d}

\def\bY{{\mathbf Y}_d}
\def\bH{{\mathbf H}}
\def\bB{{\mathbf B}}
\def\bbH{{\overline{\mathbf H}}}
\def\bhZ{{\mathbf Z}_d}
\newcommand{\bhY}{\overline{\mathbf{Y}}_d}
\newcommand{\hY}{\overline{Y}}
\def\PP{{\cal P}}
\def\CC{{\cal C}}
\def\FF{{\cal F}}

\def\balpha{\boldsymbol\alpha}

\def\bbeta{\boldsymbol\beta}
\def\bgamma{\boldsymbol\gamma}
\def\bj{\boldsymbol j}

\author{Diego Armentano\thanks{CMAT, 
Universidad de la Rep\'{u}blica, Montevideo, Uruguay. 
E-mail: diego@cmat.edu.uy.}
\quad
Jean-Marc Aza\"{i}s\thanks{IMT, UMR CNRS 5219, Universit\'e de Toulouse, 
Email: jean-marc.azais@math.univ-toulouse.fr}
\quad
Federico Dalmao\thanks{
DMEL, 
Universidad de la Rep\'{u}blica, Salto, Uruguay. 
E-mail: fdalmao@unorte.edu.uy.}
\quad
Jos\'e R. Le\'{o}n\thanks{IMERL, Universidad de la Rep\'ublica, Montevideo, Uruguay and 
Escuela de Matem\'{a}tica. Facultad de Ciencias. 
Universidad Central de Venezuela, Caracas, Venezuela. 
E-mail: rlramos@fing.edu.uy}
}

\title{Central Limit Theorem for the number of real roots of Kostlan Shub Smale random polynomial systems}
\begin{document}
\maketitle
\begin{abstract}
We obtain a Central Limit Theorem for the  
number of real roots of a square 
Kostlan-Shub-Smale random polynomial system 
of any size as the degree goes to infinity. \\

AMS Classification Primary 60F05, 30C15, Secondary 60G60, 65H10.

\end{abstract}

\tableofcontents
\section{Introduction}
The real roots of random polynomials 
have been intensively studied from the point of view of several branches of mathematics and physics. 
The investigation on this subject was initiated 
with the case of polynomials in one real variable with random coefficients 
by Bloch and P\'olya
\cite{BlPo} and  Littlewood and Offord \cite{LO1,LO2}. 
The first asymptotically sharp result on the expected number of real
roots is due to M. Kac \cite{Kac}. 
The asymptotic variance 
and a Central Limit Theorem for the number of 
real roots of random Kac polynomials 
were established by Maslova \cite{maslovavar,maslova}. 
For more details on the case of random polynomials see the textbook by Bharucha-Reid and Sambandham \cite{BR-Sam}.

The present decade witnesses a rapidly increasing series of results 
on the asymptotic distribution of the number of real roots of random polynomials. 
In 2011-2012 Granville and Wigman \cite{granville} and Aza\"{i}s and Le\'on \cite{al} established the CLT in the case of Gaussian Qualls' 
trigonometric polynomials; 
in 2016 Aza\"{i}s, Dalmao and Le\'on \cite{adl} extended this result to classical trigonometric polynomials;  
in 2015 Dalmao \cite{d} did the same for elliptic or Kostlan-Shub-Smale polynomials 
and finally 
in 2017 Do and Vu \cite{dv} proved a Central Limit Theorem for the number of real roots of Weyl polynomials. 

An important extension deals with systems of  polynomials equations. 
In Shub and Smale \cite{ss}, as suggested by Kostlan \cite{Kos}, the 
expected number of real roots or the volume of the zero level set of certain random systems of polynomials equations
has been studied for the first time. Additional cases are considered in \cite{aw0}. 
Wschebor
\cite{Wschebor} investigated the asymptotic variance of the
normalized number of real roots of the Kostlan-Shub-Smale random polynomial system 
as the dimension goes to infinity. 
The authors \cite{aadl} and 
Letendre \cite{lt}
computed the asymptotic variance of the number of roots in the square case and 
of the volume of the zero level sets of rectangular systems respectively. 

Concerning the complex version of Kostlan-Shub-Smale  polynomials 
it is worth to mention that 
Sodin and Tsirelson \cite{st} established a Central Limit Theorem 
for linear statistics of the complex zeros 
(i.e.: a sum of a test function over the set of zeros) 
using techniques closely related with our method.

\medskip

In the present paper we establish a Central Limit Theorems (CLT for short) 
for the standardized number of real roots 
of an $m\times m$ random system 
of polynomial equations as their common degree tends to infinity. 
Up to our knowledge this is the first result about the asymptotic distribution 
of the number of real roots of systems of random polynomials.

The main tool to obtain the CLT is an Hermite expansion 
of the standardized number of roots of the system. 
The main challenge is to deal with the tail of the expansion 
due to the geometry of the sphere and to some degeneracies in the covariances.   
To overcome this issue, 
we carefully construct a partition of the sphere 
such that the projections of the sets in the partition over 
the tangent spaces at their centers are asymptotically isometric. 
This construction allows us to take advantage of the existence (only locally) of a limit process. \\

\noindent\textbf{Outline of the paper:}
The paper is organized as follows. 
The main result is presented in Section 2. 
Section 3 contains an outline of the proof.  
Section 4 deals with some preliminaries. 
In Section 5 the proof of the main result is presented. 
Some technical or minor parts of the proof are postponed to Section 6.

\noindent\textbf{Some remarks on the notation:} 
We denote by $S^m$ the unit sphere in $\R^{m+1}$ and its volume by
$\kappa_m$. The variables $s$ and $t$ denote points on $S^m$ and
$ds$ and $dt$ denote the corresponding geometric measure. 
The variables $u$ and $v$ are in $\R^m$, and $du$ and $dv$ are the
associated Lebesgue measure. 
The variables $z$ and $\theta$ are reals, and $dz$ and $d\theta$ are the
associated differentials.

As usual we use the Landau's big  $O$ and small $o$ notation.
The set $\N$ of natural numbers contains $0$. 
Also, $\Const$ will denote a universal constant that may change from one line to another.

\section{Main result}

Consider a square system $\bP=0$ of $m$ polynomial equations in $m$ variables 
with common degree $d>1$. 
More precisely, 
let $\bP=(P_1,\dots,P_m)$ with 
\begin{equation*}
 P_{\ell}(t)=\sum_{|\bj|\leq d}a^{(\ell)}_{\bj}t^{\bj};\quad \ell=1,\ldots,m,
\end{equation*}
where 
\begin{enumerate}
  \item $\bj=(j_{1},\dots,j_{m})\in\N^{m}$ 
and $|\bj|=\sum^{m}_{k=1}j_{k}$; 
  \item $a^{(\ell)}_{\bj}=a^{(\ell)}_{j_1\dots j_m}\in\R$, $\ell=1,\dots,m$, $|\bj|\leq d$;
  \item $t=(t_{1},\dots,t_{m})$ 
    and $t^{\bj}=\prod^m_{k=1} t^{j_k}_{k}$.
\end{enumerate}

We say that  $\bP$ has the Kostlan-Shub-Smale 
(KSS for short) distribution if the coefficients $a^{(\ell)}_{\bj}$ 
are independent centered normally distributed random variables with variances
\begin{equation*}
  \Var\left(a^{(\ell)}_{\bj}\right)=\binom{d}{\bj}=\frac{d!}{ 
j_1!\dots j_m!(d-|\bj|)!}. 
\end{equation*}

We are interested in the number of roots of $\bP$ in $\R^m$ 
that we denote by $N_{\bP}$. 
Shub and Smale \cite{ss} proved that $\E(N_{\bP})=d^{m/2}$. 
The authors in \cite{aadl}, see also Letendre \cite{let}, proved that 
\begin{equation}\label{eq:v}
  \lim_{d\to\infty}\frac{\Var(N_{\bP})}{d^{m/2}}=V_\infty,
\end{equation}
where $0<V_\infty<\infty$. 
We now establish a CLT.
\begin{theorem}\label{tcl}
Let $\bP$ be an $m\times m$ KSS system, its standardized number of roots 
$$
\widetilde{N}_d=\frac{N_{\bP}-\E(N_{\bP})}{d^{m/4}}
$$
converges in distribution, as $d\to\infty$, towards a normal random variable with positive variance.
\end{theorem}
Our method also gives the CLT for the geometric volume of the zero level set 
of an $r\times m$ ($r<m$) KSS system. 
This will be the material of a separate note.

\section{Outline of the proof}
For the sake of readability, we present now a brief outline of the forthcoming proof. 

As a first step, it is convenient to homogenize the system. 
The roots of the original system $\bP$ are easily identified with 
the roots of the homogeneous version $\bY$ on the sphere $S^{m}$. 
Besides, the covariance structure of $\bY$ is simple and invariant 
under the action of the orthogonal group in $\R^{m+1}$. 
See the details in the next section.

In order to get the CLT 
we expand the standardized number of roots 
of $\bY$ on $S^m$ in the $L^2$-sense 
in a convenient basis, this is called an Hermite or chaotic expansion in the literature. 
Taking advantage of the structure of chaotic random variables the CLT is easily obtained 
for each term in the expansion as well as for any finite sum of them. 

The difficult part is to prove the negligeability (of the variance) 
of the tail of the expansion 
due to the degeneracy of the covariance of $\bY$ at the diagonal
$\{(s,t)\in S^m\times S^m: s=t\}$. 
To deal with this degeneracy we adapt a trick used 
under stationarity on the Euclidean case 
which consists of covering a neighbourhood of the diagonal 
with isometric small regions. 
The variance of the number of roots on each such small region is handled with Rice formula 
or some other rough method. 
Then a balance between the number of such regions and the bound is needed.

On the sphere the regions can not be chosen to be isometric, 
though a careful construction allows us to cover an essential part of the sphere 
by regions such that their orthogonal projection on the tangent space at a convenient point 
are isometric in the limit. 
The diagonal is covered by products of these regions. 
Provided the existence of a common local limit process on the tangent spaces 
we can bound uniformly the tail of the variance of the number of roots 
on each region by approximating it with the corresponding 
tail of the number of roots of the local limit process.\\

\section{Preliminaries} 
We present now some preliminaries that will be used in the sequel.

\subsection{{Homogeneous version of $\bP$}} 
Let $\bY=(Y_1,\ldots,Y_m)$ being 
\begin{equation*}
  Y_{\ell}(t)=\sum_{|\bj|=d}a^{(\ell)}_{\bj}t^{\bj},\quad \ell=1,\dots, m,
\end{equation*}
where this time  
$\bj=(j_{0},\dots,j_{m})\in\N^{m+1}$; 
$|\bj|=\sum^{m}_{k=0}j_{k}$; 
$a^{(\ell)}_{\bj}=a^{(\ell)}_{j_0\dots j_m}\in\R$; 
$t=(t_{0},\dots,t_{m})\in\R^{m+1}$ 
and 
$t^{\bj}=\prod^m_{k=0} t^{j_k}_{k}$. 
Note that $\bP(t_1,\ldots,t_m)=\bY(1,t_1,\ldots,t_m)$.

Since $\bY$ is homogeneous, 
namely for $\lambda\in\R$ it verifies $\bY(\lambda t)=\lambda^d\bY(t)$, 
its roots consist of lines through $0$ in 
$\R^{m+1}$. 
Hence, the natural place where to consider the zero sets is $S^m$ 
(or the associated projective space of $\R^{m+1}$) 
since each root of $\bP$ in $\R^m$ corresponds 
exactly to two (opposite) roots of $\bY$ on the unit sphere $S^m$ of 
$\R^{m+1}$. 
{Furthermore, one can prove that the subset of homogeneous polynomials $\bY$ 
with roots lying in the hyperplane $t_0=0$ has Lebesgue measure zero. 
Then, denoting by 
$N_{\bY}$ the number of roots of $\bY$ on $S^m$, 
we have 
$$
N_{\bP}=\frac{N_{\bY}}{2} \textrm{ almost surely}.
$$

\subsection{Angular change of variable}
We use repeatedly in the sequel that for $h:[-1,1]\to\R$ it holds that
\begin{equation}\label{eq:intS} \int_{S^m\times S^m}h(\<s,t\>)dsdt
=\kappa_m\kappa_{m-1}\int^\pi_0\sin^{m-1}(\theta)h(\cos(\theta))d\theta,
\end{equation} being $\<\cdot,\cdot\>$ the usual inner product in
$\R^{m+1}$ and $\kappa_m$ the $m$-volume of the sphere $S^m$, see
\cite[Lemma 4.2]{aadl}. 

\subsection{Covariances} 
Direct computation yields
\begin{equation*}
r_d(s,t)
:=\E({Y_{\ell}(s)Y_{\ell}(t)})
=\<s,t\>^{d};\quad s,t\in\R^{m+1}.
\end{equation*} 
As a consequence,  
the distribution of the system $\bY$ is invariant under the action of the 
orthogonal group in $\R^{m+1}$.

For $\ell=1,\ldots,m$, 
we denote by $Y'_\ell(t)$ the derivative (along the sphere) of $Y_\ell(t)$ at the point $t\in S^m$  
and by $Y'_{\ell k}$ its $k$-th component on a given basis of the tangent space of $S^m$ at the point $t$. 
We define the standardized derivative as
\begin{equation*}
 \hY_\ell'(t):=\frac{Y_\ell'(t)}{\sqrt{d}},\quad\mbox{and}\quad \bhY'(t):=\big(\hY_1'(t),\ldots,\hY'_m(t)\big).
\end{equation*}
According to the context, $\bhY'(t)$ is understood as an $m\times m$ matrix 
or as an $m^2$ vector. 
For $t\in S^m$, set also
\begin{equation}\label{eq:zeta}
  \bhZ(t)=\big(Z_1(t),\ldots,Z_{m(1+m)}(t)\big)=\big(\bY(t),\bhY'(t)\big). 
\end{equation}
The covariances 
\begin{equation}\label{eq:covZ}
\rho_{k\ell}(s,t)=\E(Z_k(s)Z_\ell(t)),\quad k,\ell=1,\ldots,m(1+m),
\end{equation}
are obtained via routine computations, see Section \ref{s:anciliary}. 
These computations are simplified using 
the invariance under isometries. 
For instance, if $k=\ell\leq m$
$$
\rho_{k\ell}(s,t)=\<s,t\>^d=\cos^d(\theta),\quad\theta\in[0,\pi),
$$where $\theta$ is the angle between $s$ and $t$.

When the indexes $k$ or $\ell$ are larger than $m$ 
the covariances involve derivatives of $r_d$. 
In fact, in \cite{aadl} is shown that $\bhZ$ is a vector of $m(1+m)$ standard normal random variables 
whose covariances depend upon the quantities 
\begin{align}\label{eq:deriv}
{\cal A}(\theta)&=-\sqrt{d}\cos^{d-1}(\theta)\sin(\theta),\\
{\cal B}(\theta)&=\cos^{d}(\theta)-(d-1)\cos^{d-2}(\theta)\sin^2(\theta),\notag\\
{\cal C}(\theta)&=\cos^{d}(\theta),\notag\\
{\cal D}(\theta)&=\cos^{d-1}(\theta),\notag
\end{align}
for $\theta\in[0,\pi)$. (See also Section \ref{s:anciliary}.
Furthermore, when dealing with the conditional distribution of $(\bhY'(s),\bhY'(t))$ 
given that $\bY(s)=\bY(t)=0$ the following expressions appear for the common variance and the correlation
\begin{equation*}
\sigma^2(\theta)=1-\frac{{\cal A}(\theta)^2}{1-{\cal C}(\theta)^2};\quad
\rho(\theta)=\frac{{\cal B}(\theta)(1-{\cal C}(\theta)^2)-{\cal A}(\theta)^2{\cal C}(\theta)}{1-{\cal C}(\theta)^2-{\cal A}(\theta)^2}.
\end{equation*}

After scaling $\theta=z/\sqrt{d}$, we have the following bounds.
\begin{lemma}[\cite{aadl}]\label{l:bounds}
There exist $0<\alpha<\frac12$ such that for $\theta=\frac{z}{\sqrt{d}}<\frac{\pi}{2}$ 
  it holds that, 
  \begin{align*}
  \left|{\cal A}(\theta)\right|&\leq z\exp(-\alpha z^2),\\
\left|{\cal B}(\theta)\right|&\leq (1+z^2)\exp(-\alpha z^2),\\
\left|{\cal C}(\theta)\right|&\leq
\left|{\cal D}(\theta)\right|\leq \exp(-\alpha z^2),\\
0\leq 1-\sigma^2&\leq \Const\cdot\exp(-2\alpha z^2),\\
|\rho|&\leq \Const\cdot(1+z^2)^2\exp(-2\alpha z^2),
\end{align*}
where $\Const$ stands for some unimportant constant, its value can change from a line to other.\qed
\end{lemma}

\subsection{ Rice formula and variance}
In \cite{aadl}, the variance $\Var(N_{\bY})$ is written as an integral over the interval $[0,\sqrt{d}\pi/2]$ 
and a domination is found in order to pass the limit wrt $d$ under the integral sign. 
More precisely, 
Rice formula, see \cite{aw}, states that
\begin{multline*}
 \Var(N_{\bY})-\E(N_{\bY})=\E(N_{\bY}(N_{\bY}-1))-(\E(N_{\bY}))^2\\
 =d^m\int_{S^m\times S^m}
 \left[\E(|\det\bhY'(s)\det\bhY'(t)|\mid \bY(s)=\bY(t)=0) p_{s,t}(0,0)\right.\\
\left.-\E(|\det\bhY'(s)|\mid \bY(s)=0)\E(|\det\bhY'(t)|\mid \bY(t)=0) p_{s}(0)p_{t}(0)\right] dsdt,
\end{multline*}
being $p_{s,t}$ the joint density of $\bY(s)$ and $\bY(t)$, and $p_s$
and $p_t$ the
densities of $\bY(s)$ and $\bY(t)$ respectively. 
The factor $d^m$ comes from the normalization of
$\bY'$ and the properties of the determinant.

By the invariance under isometries of the distribution of $\bhZ$ the integrand depends on $(s,t)\in S^m\times S^m$ 
only through $\<s,t\>$ and thus we 
can reduce the integral as in \eqref{eq:intS}.  
The conditional expectation 
$\E(|\det\bhY'(s)\det\bhY'(t)|\mid \bY(s)=\bY(t)=0)$ 
can be reduced to an ordinary expectation using the so called Gaussian regression, 
see Section \ref{s:anciliary} and \cite{aw}. 
This computations show that 
$\E(|\det\bhY'(s)\det\bhY'(t)|\mid \bY(s)=\bY(t)=0)$ and $p_{s,t}(0,0)$ depend 
on $(s,t)$ only through $\sigma^2,\rho,{\cal D}$ and ${\cal C}$ respectively. 
Hence, 
we can write
\begin{equation*}
\E(|\det\bhY'(s)\det\bhY'(t)|\mid \bY(s)=\bY(t)=0) p_{s,t}(0,0)
={\cal H}_d(\sigma^2,{\cal C},\rho,{\cal D}). 
\end{equation*}
In particular, it holds that
\begin{multline*}
\E(|\det\bhY'(s)|\mid \bY(s)=0)\E(|\det\bhY'(t)|\mid \bY(t)=0) p_{s}(0)p_{t}(0)\\
={\cal H}_d(1,0,0,0).
\end{multline*}
Thus, we can write 
\begin{multline*}
\frac{\Var(N_{\bY})-\E(N_{\bY})}{d^{m/2}}\\
=\Const\cdot d^{\frac{m-1}{2}}\int^{\sqrt{d}\pi/2}_{0}\sin^{m-1}\left(\frac{z}{\sqrt{d}}\right)
[{\cal H}_d(\sigma^2,{\cal C},\rho,{\cal D})-{\cal H}_d(1,0,0,0)]dz.
\end{multline*}
In \cite{aadl} is shown that
\begin{equation}\label{eq:cotaH}
|{\cal H}_d(\sigma^2,\mathcal C,\rho,{\cal D})-{\cal H}_d(1,0,0,0)|
\leq \Const\cdot(1-\sigma^2+|{\cal C}|+|\rho|+|\cal D|).
\end{equation}
Using the bounds in Lemma \ref{l:bounds} we obtain a domination 
for the integrand of $\Var(N_{\bY})$ in order to pass to the limit in $d$ 
under the integral sign.\\
In the same way we have the following lemma
\begin{lemma} \label{l:jm}
If $\mathcal G $ is a Borel set of $S^m$ with $m$-dimensional volume $\Vol(\mathcal G)$ 
and if $N_{\bY} ( \mathcal G)$ is the number of zeros of $\bY$ belonging to $\mathcal G $, we have
\begin{multline*}
\frac{\Var(N_{\bY}(\mathcal G))-\E(N_{\bY}(\mathcal G))}{d^{m/2}}\\
\leq \Const\cdot \Vol( \mathcal G)d^{\frac{m-1}{2}} \int^{\sqrt{d}\pi/2}_{0}\sin^{m-1}\left(\frac{z}{\sqrt{d}}\right)
|{\cal H}_d(\sigma^2,{\cal C},\rho,{\cal D})-{\cal H}_d(1,0,0,0)|dz\\
\leq \Const\cdot  \Vol( \mathcal G). 
\end{multline*}
\end {lemma}\qed

\subsection{ Wiener Chaos and Fourth Moment Theorem} We introduce now the Wiener chaos in a form that is suited to our purposes. 
For the details of this construction see \cite{pta}. 
Let $\bB=\{B(\lambda):\lambda\geq0\}$ be a standard Brownian motion 
defined on some probability space $(\Omega,\FF,\Pb)$ 
being $\FF$ the $\sigma$-algebra generated by $\bB$. 
The Wiener chaos is an orthogonal decomposition of $L^2(\bB)=L^2(\Omega,\FF,\Pb)$: 
$$
L^2(\bB)=\bigoplus^\infty_{q=0}\CC_q,
$$ 
where 
$\CC_0=\R$ and for $q\geq 1$, $\CC_q=\{I^{\bB}_q(f_q):f_q\in L^2_s([0,\infty)^q)\}$ 
being $I^{\bB}_q$ the $q$-folded multiple integral wrt $\bB$ and $L^2_s([0,\infty)^q)$ 
the space of kernels $f_q:[0,\infty)^q\to\R$ which are square integrable and 
symmetric, that is, if $\pi$ is a permutation then 
$f_q(\lambda_1,\dots,\lambda_q)=f_q(\lambda_{\pi(1)},\dots,\lambda_{\pi(q)})$.  
Equivalently, each square integrable functional $F$ of the Brownian motion $\bB$ 
can be written as a sum of orthogonal random variables
\begin{equation*}
F=\E(F)+\sum^\infty_{q=1}I^{\bB}_q(f_{q}),
\end{equation*}
for some uniquely determined kernels $f_q\in L_s^2([0,\infty)^q)$.

We need to introduce the so-called contractions of the kernels. 
Let $f_q,g_q\in L_s^2([0,\infty)^q)$, 
then for  $n=0,\dots,q$ 
we define $f_q\otimes_ng_q\in L_s^2([0,\infty)^{2q-2n})$ given
by
\begin{multline}\label{eq:cont}
f_q\otimes_ng_q(\lambda_1,\dots,\lambda_{2q-2n})\\
=\int_{[0,\infty)^n}f_q(z_1,\dots,z_n,\lambda_1,\dots,\lambda_{q-n})\\
\cdot g_q(z_1,\dots,z_n,\lambda_{q-n+1},\dots,\lambda_{2q-2n})dz_1\dots dz_n.
\end{multline}

Now, we can state the generalization of the Fourth Moment Theorem.
\begin{theorem}[\cite{pta} Theorem 11.8.3]\label{teo:4o}
Let $F_d$ be in $L^2(\bB)$ admit chaotic expansions
\begin{equation*}
F_d=\E(F_d)+\sum^\infty_{q=1}I_q(f_{d,q})
\end{equation*}
for some kernels $f_{d,q}$. 
Then, if $\E(F_d)=0$ and
\begin{enumerate}
\item for each fixed $q\geq 1$, $\lim_{d\to\infty}\Var(I_{q}(f_{d,q}))=V_q$;
\item $V:=\sum^\infty_{q=1}V_q<\infty$;
\item for each $q\geq 2$ and $n=1,\dots,q-1$, 
$$\lim_{d\to\infty}\|f_{d,q}\otimes_n
      f_{d,q}\|_{L_s^2([0,\infty)^{2q-2n})}=0;$$
\item $\lim_{Q\to\infty}\limsup_{d\to\infty}\sum^\infty_{q=Q+1}\Var(I_{q}(f_{d,q}))=0$.
  \end{enumerate}
Then, $F_d$ converges in distribution towards the $N(0,V)$ distribution. \qed
\end{theorem}
Condition 1,2 and 4 are variance conditions. 
Condition 3 is a moment condition 
or equivalently a condition on the decay of tail of the the density function. 
It is ultimately written in terms of the covariances of the process $\bhZ$ 
as in Theorem 7.2.4 of \cite{np}, see Lemma \ref{lemma:cont-cov}.\\

\subsection{ Hermite expansion of $N_{\bY}$} 

The Hermite expansion of $N_{\bY}$ was obtained in \cite{aadl}, see also \cite{let}. 

We introduce the Hermite polynomials $H_n(x)$ by $H_0(x)=1$, $H_1(x)=x$ and 
$H_{n+1}(x)=xH_n(x)-nH_{n-1}(x)$. 
The multi-dimensional (tensorial) versions are, 
for multi-indexes 
$\balpha=(\alpha_\ell)\in\N^m$ and 
$\bbeta=(\beta_{\ell,k})\in\N^{m^2}$, 
and vectors ${\mathbf y}=(y_\ell)\in\R^m$ 
and ${\mathbf y}'=(y'_{\ell,k})\in\R^{m^2}$
$$
\bH_{\balpha}(\mathbf y)=\prod^m_{\ell=1} H_{\alpha_\ell}(y_\ell),\quad  
\bbH_{\bbeta}(\mathbf y')=\prod^{m}_{\ell,k=1} 
H_{\beta_{\ell,k}}(y'_{\ell,k}).
$$ 
It is well known that the standardized Hermite polynomials 
$\{\frac1{\sqrt{n!}}H_n\}$, $\{\frac1{\sqrt{\balpha!}}\bH_{\balpha}\}$ 
and $\{\frac1{\sqrt{\bbeta!}}\overline{\bH}_{\bbeta}\}$ 
form orthonormal bases of the spaces $L^2(\R,\phi_1)$, $L^2(\R^m,\phi_m)$ and $L^2(\R^{m^2},\phi_{m^2})$ respectively. 
Here, $\phi_j$ stands for the standard Gaussian measure on $\R^j$, and $\balpha!=\prod^m_{\ell=1}\alpha_\ell !$, $\bbeta!=\prod^m_{\ell,k=1}\beta_{\ell,k}!$. 
Sometimes we write $\bbeta=(\bbeta_1,\ldots,\bbeta_m)$ with $\bbeta_\ell=(\beta_{\ell 1},\ldots,\beta_{\ell m})\in\N^m$ 
and $\bbeta!=\prod^m_{\ell=1}\bbeta_\ell!$. 
See \cite{np,pta} for a general picture of Hermite polynomials.

Let $f_{\bbeta}$ ($\bbeta\in\R^{m^2}$) 
be the coefficients in the Hermite's basis 
of the function 
$f:\R^{m^2}\to\R$ in $L^2(\R^{m^2},\phi_{m^2})$ defined by
\begin{equation}\label{eq:det}
f({\mathbf y'})=|\det(\mathbf y')|. 
\end{equation}
That is $f({\mathbf y'})=\sum_{\bbeta\in\R^{m^2}}f_{\bbeta}\overline{\bH}_{\bbeta}({\mathbf y'})$ with 
\begin{equation}\label{eq:fbeta}
f_{\bbeta}
=\frac1{\bbeta!}\int_{\R^{m^2}}|\det(\mathbf y')|\overline{\bH}_{\bbeta}(\mathbf y')\phi_{m^2}(\mathbf y')d\mathbf y'.
\end{equation}

Parseval's Theorem entails
$||f||^2_2=\sum_{q=0}^\infty \sum_{|\bbeta|=q}f_{\bbeta}^2\bbeta!<\infty$. 
Moreover, 
since the function $f$ is even w.r.t. each column, the above coefficients are zero whenever 
$|\bbeta_\ell|$ is odd for at least one $\ell=1,\ldots,m.$

Now, consider the coefficients in the Hermite's basis in 
$L^2(\R,\phi_1)$
for the Dirac delta $\delta_0(x)$.  
They are $b_{2j}=\frac1{\sqrt{2\pi}}(-\frac12)^j\frac1{j!},$ 
and  zero for odd indices, see \cite{kl-97}. Introducing now the distribution $\prod_{j=1}^m\delta_0(y_j)$ and denoting by $b_{\balpha}$ its coefficients 
it holds 
\begin{equation}\label{eq:balpha}
b_{\balpha}=\frac1{[\frac{\balpha }2]!}\prod_{j=1}^m\frac1{\sqrt{2\pi}}\bigg[-\frac12\bigg]^{[\frac{\alpha_j}2]} 
\end{equation}
or $b_{\balpha}=0$ if at least one index $\alpha_j$ is odd.

\begin{proposition}[\cite{aadl}Proposition 3.3]\label{prop:expansion}
With the same notations as above, we have, in the $L^2$ sense, that 
\begin{equation*}
\widetilde{N}_d:=\frac{N_{\bY}-2d^{m/2}}{2d^{m/4}}=\sum^\infty_{q=1}I_{q,d},
\end{equation*}
where
\begin{equation*}
I_{q,d}=\frac{d^{m/4}}{2}\int_{S^{m}}
\sum_{|\bgamma|=q}c_{\bgamma}\widetilde{\bH}_{\bgamma}(\bhZ(t))
dt,
\end{equation*}
where $\bgamma=(\balpha,\bbeta)\in\N^m\times\N^{m^2}$, 
$|\bgamma| = |\alpha|+|\beta|$, $c_{\bgamma}=b_{\balpha}f_{\bbeta}$ and 
\begin{equation}\label{eq:defHtilde}
    \widetilde{\bH}_{\bgamma}(\mathbf
    Z):=\bH_{\balpha}(\mathbf{Y)\bbH_{\bbeta}(\mathbf{Y'}}),
\end{equation}
  for $\mathbf Z=(\mathbf Y,\mathbf{Y'})\in\R^{m}\times\R^{m^2}$.\qed
\end{proposition}
In Lemma \ref{lemma:kernels} $I_{q,d}$ is written as a stochastic
integral with respect to the Brownian motion.

\begin{remark}\label{r:subset}
A similar expansion holds for the number of roots $N_{\bY}(\mathcal{G})$ of $\bY$  
on a Borel subset $\mathcal G$ of $S^{m}$. 
In fact, 
in order to obtain the expansion of $N_{\bY}(\mathcal{G})$ 
each factor in the integrand in Kac formula \cite{aw}
$$
N_{\bY}(\mathcal{G})
=\lim_{\delta\to0}\int_{\mathcal{G}}|\det\bY'(t)|\cdot\frac{1}{(2\delta)^m}\indicator_{[-\delta,\delta]^m}(\bY(t))dt,
$$
is expanded in the Hermite basis.  
Then, one needs to take the sums out of the integral sign. 
We have
\begin{equation*}
\widetilde{N}_d(\mathcal{G}):=\frac{N_{\bY}(\mathcal{G})-\E(N_{\bY}(\mathcal{G}))}{d^{m/4}}=\sum^\infty_{q=1}I_{q,d}(\mathcal{G}),
\end{equation*}
with
\begin{equation*}
I_{q,d}(\mathcal{G})=\frac{d^{m/4}}{2}\int_{\mathcal{G}}
\sum_{|\bgamma|=q}c_{\bgamma}\widetilde{\bH}_{\bgamma}(\bhZ(t))
dt.
\end{equation*}
\end{remark}

\section{Proof of Theorem \ref{tcl}}
Now, let us verify the conditions in Theorem \ref{teo:4o} for
$\widetilde{N}_d$.

Define
\begin{equation}\label{eq:Gq}
G_q(\mathbf
  z)=\sum_{|\bgamma|=q}c_{\bgamma}\widetilde{\bH}_{\bgamma}(\mathbf z),
\end{equation}
where $\widetilde{\bH}$ is given in (\ref{eq:defHtilde}), 
so that $I_{q,d}=\frac{d^{m/4}}{2}\int_{S^m}G_q(\bhZ(t))dt$. 
Mehler's formula, see Lemma 10.7 in \cite{aw}, shows that 
$\E\left[\widetilde{\bH}_{\bgamma}(\bhZ(s))\widetilde{\bH}_{\bgamma'}(\bhZ(t))\right]$
can be written as a linear combination of powers of the covariances $\rho_{k\ell}$ 
of the process $\bhZ$ which depend on $s,t$ only through $\< s,t\>$. 
Hence, we can define
\begin{equation}\label{e:H}
{\mathcal H}_{q,d}\left(\<s,t\>\right)
=\E(G_q(\bhZ(s))G_q(\bhZ(t))).
\end{equation}
Lemma \ref{l:Hpar} in the Section \ref{s:anciliary} show that ${\mathcal H}_{q,d}$ is an even function.

\subsection{Convergence of partial sums}
In this section we prove points 1,2 and 3 in Theorem \ref{teo:4o}.

\noindent {\sc Point 1.} 
We compute the variance of the term corresponding to a fixed $q$. 
We have
\begin{equation*}
\Var(I_{q,d})=\frac{d^{m/2}}{4}\Var\left(\int_{S^m}G_q(\bhZ(t))dt\right)
=\frac{d^{m/2}}{4}
\int_{S^m\times S^m}
{\mathcal H}_{q,d}\left(\< s,t\>\right)dsdt.
\end{equation*}
As above,  the invariance 
under isometries of the distribution of $\bhZ$ 
and \eqref{eq:intS} we get
\begin{multline*}
\Var(I_{q,d}) 
=\kappa_m\kappa_{m-1}\frac{d^{m/2}}{4}\int_0^{\pi}\sin^{m-1}(\theta)
{\mathcal H}_{q,d}\left(\cos(\theta)\right) d\theta\\
=\frac{\kappa_m\kappa_{m-1}}2\int_0^{\sqrt d\pi/2}d^{(m-1)/2}
\sin^{m-1}\left(\frac{z}{\sqrt{d}}\right){\mathcal H}_{q,d}\left(\cos\left(\frac{z}{\sqrt{d}}\right)\right) dz.
\end{multline*}
In the last equality we used the change of variables $\theta\mapsto
\pi-\theta$ on the interval $[\pi/2,\pi]$, 
the scaling $\theta= z/\sqrt{d}$ and the fact that ${\cal H}_{q,d}$ is even, see Lemma \ref{l:Hpar}.

The convergence follows by  dominated convergence using for the
covariances $\rho_{k\ell}=\E(Z_k(s)Z_\ell(t))$, given in \eqref{eq:covZ},
the bounds in Lemma \ref{l:bounds} and the domination for ${\cal
H}_d=\sum^{\infty}_{q=0}{\cal H}_{q,d}$ given in \eqref{eq:cotaH}.

\vspace{10pt}

\noindent {\sc Point 2.} 
Recall from \eqref{eq:v} that
$$
V_\infty=\lim_{d\to\infty}\frac{\Var(N_{\bY})}{d^{m/2}}
=\lim_{d\to\infty}\sum^{\infty}_{q=0}\Var(I_{q,d}).
$$ 
The second equality follows from Parseval's identity. 
Thus, by Fatou's Lemma 
$$
\sum^{\infty}_{q=0}\lim_{d\to\infty}\Var(I_{q,d})
\leq V_\infty
<\infty.
$$
Actually, 
equality holds as a consequence of Point 4 and the finiteness of $V_\infty$.        
\vspace{10pt}

\noindent {\sc Point 3.}  
Next Lemma, which proof is postponed to Section \ref{s:kernels}, 
gives a sufficient condition on the covariances of the process $\bhZ$
in order to verify the convergence of the norm  of the contractions. 
Recall that the law of the process is invariant under isometries, 
$r_d(s,t)=r_d(\<s,t\>)$, thus, $r_d$ can be seen as a function of one real variable. 

Let $g_{q,d}\in L^2_s([0,\infty)^q)$ be such that $I_{q,d}=I^{\bB}_q(g_{q,d})$, 
see Lemma \ref{lemma:kernels}. 
\begin{lemma}\label{lemma:cont-cov}
For $k=0,1,2$, 
  let $r^{(k)}_d$ indicate the $k$-th derivative of $r_d:[-1,1]\to\R$. 
If
\begin{equation}\label{eq:cont-cov}
    \lim_{d\to +\infty}
d^{m/3}\int^{\pi/2}_{0}
\sin^{m-1}(\theta)\,
|r^{(k)}_d(\cos(\theta))|\,
  d\theta =0,
\end{equation}
then, for $n=1,\dots,q-1$:
$$
\lim_{d\to\infty}\|g_{q,d}\otimes_n g_{q,d}\|_2=0.
$$
\end{lemma}
Therefore, 
it suffices to verify \eqref{eq:cont-cov}. 
For $k=0,1,2$ we have
\begin{multline*}
d^{m/3}\int^{\pi/2}_{0}
\sin^{m-1}(\theta)
|r^{(k)}_d(\cos(\theta))|
d\theta\\
=d^{m/3}
\int^{\sqrt{d}\pi/2}_{0}
\sin^{m-1}\left(\frac{z}{\sqrt{d}}\right)
\left|r^{(k)}_d\left(\cos\left(\frac{z}{\sqrt{d}}\right)\right)\right|
\frac{dz}{\sqrt{d}}\\
=
\frac{1}{d^{m/6}}
\int^{\sqrt{d}\pi/2}_{0}
d^{\frac{m-1}{2}}\sin^{m-1}\left(\frac{z}{\sqrt{d}}\right)
\left|r^{(k)}_d\left(\cos\left(\frac{z}{\sqrt{d}}\right)\right)\right|dz.
\end{multline*}
Now $d^{\frac{m-1}{2}}\sin^{m-1}\left(z/\sqrt{d}\right)\leq z^{m-1}$ 
and taking the worst case in Lemma \ref{l:bounds} 
we have $|r^{(k)}_d(z/\sqrt{d})|\leq (1+z^2)\exp(-\alpha z^2)$. 
Hence, the last integral is convergent and (\ref{eq:cont-cov}) follows.

\subsection{The tail is negligeable}
In this section we deal with Point 4 in Theorem \ref{teo:4o}. 
Let $\pi_q$ be the projection on the $q$-th chaos $\CC_q$ 
and $\pi^Q=\sum_{q\geq Q}\pi_q$ be the projection 
on $\oplus_{q\geq Q}\CC_q$. 
We need to bound the following quantity uniformly in $d$
\begin{eqnarray}\label{tailvariance}
\frac{d^{m/2}}4\Var(\pi^Q(N_{\bY}))
=\frac14\sum_{q\geq Q} d^{m/2}\int_{S^m\times S^m}{\mathcal H}_{q,d}(\<s,t\>)dsdt,
\end{eqnarray}
where ${\mathcal H}_{q,d}$ is defined in \eqref{e:H}.

In order to bound this quantity we split the integral into two parts depending on the distance between $s$ and $t$. 
The (geodesical) distance between $s,t\in S^m$ is defined as 
\begin{equation}\label{eq:geodist}
\dist(s,t)=\arccos(\left\langle s,t\right\rangle).
\end{equation}
We divide the integral into the integrals over the regions
$\{(s,t):\dist(s,t)<a/\sqrt{d}\}$ and its complement, 
where $a$ will be chosen later. 
We do this in the following two subsections.

\subsubsection{Off-diagonal term}
In this subsection we consider the integral in the rhs of \eqref{tailvariance} 
restricted to the off-diagonal region $\{(s,t):\dist(s,t)\geq a/\sqrt{d}\}$. 
That is,
$$
\frac{d^{m/2}}{4}\sum_{q\geq Q}\int_{\{(s,t):\dist(s,t)\geq
a/\sqrt{d}\}}{\mathcal H}_{q,d}(\<s,t\>)dsdt.
$$
This is the easier case since the covariances of $\bhZ$ are bounded away from $1$. 

Before continuing our proof we need the following lemma from Arcones (\cite{arcones}, page 2245). 
Let $X$ be a standard Gaussian vector on $\R^N$ and $h:\R^N\to\R$ a measurable function 
such that $\mathbb E[h^2(X)]<\infty$, and let us consider its $L^2$ convergent Hermite's expansion 
$$
h(\mathbf x)=\sum_{q=0}^\infty\sum_{|\mathbf k|=q}h_{\mathbf
k}H_{\mathbf k}(\mathbf x),\quad \mathbf x\in\R^N.
$$
The Hermite rank of $h$ is defined as 
$$\mbox{rank}(h)=\inf\{\tau:\, \exists\; \mathbf k\in\N^N\,, |\mathbf k|=\tau\,; \E[(h(X)-\E h(X))H_{\mathbf k}(X)]\neq0\}.$$
Then, we have the following result.
\begin{lemma}[\cite{arcones}]
Let $W=(W_1,\ldots,W_N)$ and $Q=(Q_1,\ldots,Q_N)$ be two mean-zero Gaussian random vectors on $\R^N$. Assume that
\begin{equation*}
\mathbb E[W_jW_k]=\mathbb E[Q_jQ_k]=\delta_{j,k}, 
\end{equation*}
for each $1\le j,k\le N$. 
We define 
\begin{equation*}
r^{(j,k)}=\mathbb E[W_jQ_k].
\end{equation*}
Let $h$ be a function on $\R^N$ with finite second moment and Hermite rank $\tau,$ $1\le\tau<\infty,$ define
\begin{equation*}
\psi:= \max\left\{\max_{1\le j\le N}\sum_{k=1}^N|r^{(j,k)}|,\max_{1\le k\le N}\sum_{j=1}^N|r^{(j,k)}|\right\}.
\end{equation*}
Then
\begin{equation*}
|\Cov(h(W),h(Q))|\le \psi^{\tau}\mathbb E[h^2(W)].
\end{equation*}
\qed
\end{lemma}
We apply this lemma for $N=m(1+m)$, $W=\mathbf Z(s)$, $Q=\mathbf Z(t)$ 
and to the function $h(\mathbf x)=G_q(\mathbf x)$, defined in \eqref{eq:Gq}. Recalling that 
$\rho_{k\ell}(s,t)=\rho_{k\ell}(\<s,t\>)=\E[Z_k(s)Z_\ell(t)]$, 
the Arcone's coefficient is now 
$$\psi(s,t)=\max\left\{\sum_{1\le k\le m(1+m)}|\rho_{k\ell}(s,t)|,\sum_{1\le \ell\le m(1+m)}|\rho_{k\ell}(s,t)|\right\}.$$
Thus
$$|{\cal H}_{q,d}(\<s,t\>)|\le \psi(\<s,t\>)^q||G_q||^2,$$
being 
$\|G_q\|^2=\E(G^2_q(\zeta))$ for standard normal $\zeta$.
\begin{lemma}\label{l:normaG}
For $f$ and $G_q$ defined in \eqref{eq:det} and \eqref{eq:Gq} respectively, 
it holds
$$
||G_q||^2\le ||f||^2_2.
$$
\end{lemma}
The proof is postponed to Section \ref{s:anciliary}. 

\medskip 

We move to the bound of Arcones' coefficient $\psi(\<s,t\>)$. 
By the invariance of the distribution of $\bY$ (and $\bhZ$) under isometries 
we can fix $s=e_0$ and express 
$\<e_0,t\>=\cos(\theta)$ being $\theta$ the angle between $e_0$ and $t$ as above.  
Direct computation of the covariances $\rho_{k\ell}$, see Section
\ref{s:anciliary}, yields that 
the maximum in the definition of $\psi$ is $|\mathcal C|+\mathcal |A|$,
(see \eqref{eq:deriv}).
Lemma \ref{l:bounds} 
entails that $|\mathcal C|+|\mathcal A|\le e^{-\alpha z^2}(1+z).$  
For $z=2$ the bound takes the value
$2e^{-4\alpha}$ which is less or equal to one if $\alpha\ge \frac14\log2$, this is always possible because the only restriction that we have is $\alpha<\frac12.$ 
Moreover, for $\delta$ small enough $e^{-\alpha z^2}(1+z)\ge1$ if $z<\delta.$ 
This leads to affirm that there exists an $a<2$ such that for all $z\ge a$ it holds $|\mathcal C|+|\mathcal A|<r_0< 1$.

These results  allow to use the Arcones' result to obtain
\begin{multline*}
\sup_{d}\sum_{q\geq Q}\frac{d^{m/2}}4
\int_{\left\{(s,t):\dist(s,t)\geq \frac{a}{\sqrt{d}}\right\}}{\mathcal H}_{q,d}(\<s,t\>)dsdt\\
=\sup_{d}\frac{C_m}4\left|\sum_{q\geq Q} d^{\frac{m-1}2}
\int_{a}^{\sqrt d\pi}\sin^{m-1}\left(\frac z{\sqrt d}\right)
{\cal H}^{q}_d\left(\cos\left(\frac z{\sqrt d}\right)\right)dz\right|\\
\le C_m||f||_2^2\sum_{q\geq Q} r^{q-1}_0\int_{a}^\infty
  z^{m-1}(1+z)e^{-\alpha z^2}dz.
\end{multline*}
Therefore we conclude
$$
\lim_{Q\to\infty} 
\sup_{d}\sum_{q\geq Q}\frac{d^{m/2}}4
\int_{\left\{(s,t):\dist(s,t)\geq \frac{a}{\sqrt{d}}\right\}}{\mathcal H}_{q,d}(\<s,t\>)dsdt\\
=0,
$$
obtaining the result for the restriction to the off-diagonal term.

\subsubsection{Diagonal term}
In this subsection we prove that the integral in the rhs of \eqref{tailvariance} 
restricted to the diagonal region $\{(s,t):\dist(s,t)< a/\sqrt{d}\}$
tends to $0$ as $Q$ tends to $\infty$ uniformly in $d$, 
where $a<2$ remains fixed. 
That is, we consider
$$
\frac{d^{m/2}}{4}\sum_{q\geq Q}\int_{\{(s,t):\dist(s,t)< a/\sqrt{d}\}}{\cal H}_{q,d}(\<s,t\>)dsdt.
$$
This is the difficult part, we use an indirect argument.

We introduce now the hyperspherical coordinates and sphere asymptotic partition.

The hypherspherical coordinates of the sphere are given in the following
way.
For $\theta=(\theta_1,\dots,\theta_{m-1},\theta_m)\in[0,\pi)^{m-1}\times[0,2\pi)$ 
we write 
$$
x^{(m)}(\theta)=(x^{(m)}_1(\theta),\dots,x^{(m)}_{m+1}(\theta))\in S^m,
$$ 
where
\begin{align}\label{e:hs}
 x^{(m)}_k(\theta)  &=
  \prod^{k-1}_{j=1}\sin(\theta_j)\cdot\cos(\theta_k),\quad k\leq m\notag\\
  x^{(m)}_{m+1}(\theta)  &= \prod^{m}_{j=1} \sin(\theta_j);
\end{align}
with the convention that $\prod^0_1=1$.

Next proposition gives a convenient partition of the sphere based
on this coordinates. 

Define the hyperspherical rectangle (HSR for short) with 
center $x^{(m)}(\tilde{\theta})$ with $\tilde{\theta}=(\tilde{\theta}_1,\ldots,\tilde{\theta}_m)$ 
and vector radius $\tilde{\eta}=(\tilde{\eta}_1,\ldots,\tilde{\eta}_m)$ as
\begin{equation*}
HSR(\tilde{\theta},\tilde{\eta})=
\{x^{(m)}(\theta):|\theta_i-\tilde{\theta}_i|<\tilde{\eta}_i,i=1,\ldots,m\}.
\end{equation*}
Let $T_tS^m$ be the the tangent space to $S^m$ at $t$. This space  can be identified with $t^\bot\subset\R^{m+1}$.   
Let $\phi_t:S^m\to t^\bot$ be the orthogonal projection over $t^{\bot}$. 
The details and the proof are presented in Section \ref{s:partition}.
\begin{proposition}\label{p:partition}
For $d$ large enough, 
there exists a partition of the unit sphere $S^m$ into hyperspherical rectangle $R_j:j=1,\ldots,k(m,d)=O(d^{m/2})$ 
  and an extra set $E\subset S^m$ such that
\begin{enumerate}
 \item $\Var(N_{\bY}(E))=o(d^{m/2})$.
 \item The HSRs $R_j$ have diameter $O(\frac{1}{\sqrt{d}})$ and if $R_j$ and $R_\ell$ do not share any border point 
 (they are not neighbours), then $\dist(R_j,R_\ell)\geq \frac{1}{\sqrt{d}}$. 
 \item The projection of each of the sets $R_j$ on the tangent space at its center $c_j$, 
after normalizing by the multiplicative factor $\sqrt{d}$, converges to
    the rectangle $ [-1/2, 1/2]^m$ in the sense of Hausdorff
    distance. 
\end{enumerate}
\end{proposition}

Set $r=d^{-1/2}$. 
We can write 
$S^m=\bigcup^{k(m,r)}_{j=1}R_j\cup E$,
and 
\begin{multline*}
\pi^Q(N_{\bY})=\sum_{q\geq Q} \int_{S^m}G_q(\bhZ(t))dt\\
=\sum_{q\geq Q}\left[\sum_j\int_{R_{j}}G_q(\bhZ(t))dt+\int_EG_q(\bhZ(t))dt\right]. 
\end{multline*}
By the first item in Proposition \ref{p:partition} and \eqref{e:H} we have
$$\E\left[(\pi^Q(N_{\bY}))^2\right]=\sum_j\sum_\ell\sum_{q\geq Q} 
\int_{R_{j}}\int_{R_{\ell}}
{\mathcal H}_{q,d}(\<s,t\>)
dsdt+o(d^{m/2}).
$$

Actually, in this section we are interested in covering a strip around
the diagonal $\{(s,t)\in S^m\times S^m:\dist(s,t)<ar\}$, $a<2$.  Hence,
we restrict the sum in the last equation to the set
$\{(j,\ell):|j-\ell|\leq 2\}$.  Clearly the number of sets verifyng this
condition is $O(r^{-1})=O(\sqrt{d})$ and below we prove that the tail of
the variance of $N_{\bY}(R_j)/d^{m/2}$ is 
$O(d^{-m/2})$ and the implicit constant in the $O$ notation 
does not depend on $j$.  
Therefore, it remains to bound
\begin{multline*}
\sum_{(j,\ell):|j-\ell|<2} 
\int_{R_{j}}\int_{R_{\ell}}\sum_{q\geq Q}{\mathcal H}_{q,d}(\<s,t\>)dsdt\\
=\sum_{(j,\ell):|j-\ell|<2} 
\E\left[\int_{R_{j}}\sum_{q\geq Q}G_q(\bhZ(s))ds\cdot
\int_{R_{\ell}}\sum_{q\geq Q}G_q(\bhZ(t))dt\right]
\\
\leq\sum_{(j,\ell):|j-\ell|<2} \left[\sum_{q\geq Q}\int_{R_{j}\times
  R_j}{\mathcal H}_{q,d}(\<s,t\>)dsdt\right]^{1/2} \left[\sum_{q\geq
  Q}\int_{R_{\ell}\times R_\ell}{\mathcal
  H}_{q,d}(\<s,t\>)dsdt\right]^{1/2},
\end{multline*}
where the inequality follows by Cauchy-Schwarz. 
Fix $j$, in order to bound 
$$
\sum_{q\geq Q}\int_{R_{j}\times R_j}{\mathcal H}_{q,d}(\<s,t\>)dsdt, 
$$
we note that it coincides with $\Var\left(\pi^Q\left(N_{\bY}(R_j)\right)\right)$ .

On the other hand, we prove hereunder  that there exist some local limit 
for $\bY$ as $d\to\infty$.

At this point it is convenient to work with caps
$$
C(s_0,\gamma r)=\{s:d(s,s_0)<\gamma r\}.
$$
Note that by the second item in Proposition \ref{p:partition}, each HSR $R_j$ is included in a cap of radius $\gamma r$ 
for some $\gamma$ depending on $m$. 

By the invariance under isometries of the distribution of $\bY$, 
the distribution of the number of roots on a cap $C(s_0,\gamma r)$ does not depend on its center $s_0$. 
Thus, without loss of generality we work with the cap of angle $\gamma r$ centered at the east-pole $e_{0}=(1,0,\ldots,0)$, 
$$
C\left(e_{0},\gamma r\right)
=\left\{t\in S^m:\,\dist(t,e_{0})<\gamma r\right\}.$$ 
(See Nazarov-Sodin \cite{ns}.)

We use the local chart $\phi:C(e_{0},\gamma r)\to B(0,\sin(\gamma r))\subset \R^m$  
defined by 
$$
\phi^{-1}(u)=(\sqrt{1-\|u\|^2},u),\quad u\in B(0,\sin(\gamma r)),
$$ 
to project this set over the tangent space. 

Define the random field 
$\mathcal Y_d:B\left(0,\gamma \right)\subset\R^m\to\R^m$, as 
$$\mathcal Y_d(u)=\bY(\phi^{-1}(u/r)).$$ 
Observe that the $\ell$ coordinates, ${\cal Y}^{(\ell)}_d$ say, of ${\cal Y}_d$ are independent. 
Clearly the number of roots of $\bY$ on $R\subset C(e_{0},\gamma r)$ 
and the number of roots of ${\mathcal Y}_d$ on $\phi(R/r)\subset B(0,\gamma)$ coincide. 
That is
$$
N_{\bY}(R)=N_{{\mathcal Y}_d}(\phi(R/r)).
$$

\begin{proposition}\label{p:local}
The sequence of processes ${\mathcal Y}^{(\ell)}_d(u)$ 
and its first and second order derivatives 
converge in the finite dimensional distribution sense towards the mean zero Gaussian processes $\mathcal Y_{\infty}$ 
with covariance function $\Gamma(u,v)=e^{-\frac{||u-v||^2}2}$ 
and its corresponding derivatives. 
\end{proposition}
\begin{proof} We give a short proof  for completeness,  see   also \cite{ns}. 
The covariance of $Y_\ell(s)$ and $Y_\ell(t)$ is $\<s,t\>^d,$ whenever $s,t\in S^m$. 
In this form  we get for $\phi^{-1}(u), \phi^{-1}(v)\in S^m$
$$\<\phi^{-1}(u),\phi^{-1}(v)\>=\sum_{i=1}^m u_iv_i+\sqrt{1-||u||^2}\sqrt{1-||v||^2}.$$
Using the rescaling we have
$$\<\phi^{-1}\left(\frac u{\sqrt d}\right),\phi^{-1}\left(\frac v{\sqrt d}\right)\>
=\frac1d\sum_{i=1}^m u_iv_i + 
\sqrt{1-\left\|\frac u{\sqrt d}\right\|^2} \sqrt{1-\left\|\frac v{\sqrt d}\right\|^2}.
$$ 
The Taylor development for $\sqrt{1-x^2}$ gives
\begin{multline*}
\E\left({\cal Y}^{(\ell)}_d\left(u\right){\cal Y}^{(\ell)}_d\left(v\right)\right)=
\<\phi^{-1}\left(\frac u{\sqrt d}\right),\phi^{-1}\left(\frac v{\sqrt d}\right)\>^d\\
=\left(1-\frac{||u-v||^2}{2d}+O\left(\frac{||u||^4}{d^2}+\frac{||v||^4}{d^2}\right)\right)^d
\mathop{\to}_{d\to\infty} e^{-\frac{||u-v||^2}2}. 
\end{multline*}
The convergence is uniform over compacts, 
thus the partial derivatives of this function also converge, towards the derivatives of the limit covariance. Then the claimed result holds in force.
\end{proof}
\begin{remark}

Using classical criteria on the fourth moment of increments, 
the  weak convergence in the space on continuous functions can be proved. 
But we do not need it in the sequel. 
\end{remark}

\begin{remark}
  The local limit process $\mathcal Y_\infty$ has as coordinates
  $\mathcal Y^{(\ell)}_\infty$, each of one is an
independent copy of the random field with  covariance
$$
  \Gamma(u)=e^{-\frac{||u||^2}2},\quad  u\in\R^m.
$$
 Then its covariance matrix writes 
$$\tilde \Gamma(u)= \diag(\Gamma(u),\ldots,\Gamma(u)).$$
The second derivative matrix  $\tilde \Gamma''(u)$ 
can be written in a similar way, but here the blocks are equal to the matrix
$\Gamma''(u)=(a_{ij})$ where $a_{ii}=e^{-\frac{||u||^2}2}H_2(u_i)$ 
and $a_{ij}=e^{-\frac{||u||^2}2}H_1(u_i)H_1(u_j)$ if $i\neq j$. 
We can adapt  the Estrade and Fournier \cite{ef} result that says that the second moment of the roots in a compact set of such a process exists if for  some $\delta>0$ we have
$$
\int_{B(0,\delta)}\frac{||\tilde\Gamma''(u)-\tilde
  \Gamma''(0)||}{||u||^m}du=m\int_{B(0,\delta)}\frac{||\Gamma''(u)-I||}{||u||^m}du<\infty.
$$ 
Since $\Gamma$ is $C^\infty$ we have 
$$||\Gamma''(0)-\Gamma''(u)||=o(||u||)\mbox{ as } u\to0,$$
The convergence of the above integral follows easily using hyperspherical coordinates.
\end{remark}

A key fact is that the local limit process, 
though it can not be defined globally, 
has the same distribution regardless $j$. 
Thus, we bound $\Var\left(\pi^Q\left(N_{\bY}(R_j)\right)\right)$ uniformly in $j$ 
by approximating it with the tail of the variance of the number of zeros of the limit process $\mathcal Y_{\infty}$ 
on the limit set $[-1/2,1/2]^m$.

\begin{proposition}\label{p:cap}
For all $j\leq k(m,d)$ and $\varepsilon>0$ there exist $d_0$ and $Q_0$ such that for $Q\geq Q_0$
$$
\sup_{d>d_0}\E[(\pi^Q(N_{\bY}(R_j)))^2]<\varepsilon.
$$
\end{proposition}
\begin{proof}
Let $R=R_j\subset C(e_{0},\gamma r)$, 
By Remark \ref{r:subset}, 
the Hermite expansion holds true also for the number of roots of $\bY$ on any subset of $S^m$. 
Hence, 
$$
N_{\bY}(R)=\sum_{q=0}^\infty d^{\frac m2}\int_{R}G_q(\bhZ(t))dt.
$$ 
Let us define $\tilde R=\phi(R)\subset B(0,\sin\frac a {\sqrt d})\subset \R^m$. 
It follows that
$$
N_{{\mathcal Y}_d}(\tilde R)=N_{\bY}(R)
=\sum_{q=0}^\infty d^{\frac m2}\int_{\tilde R}G_q({\mathcal
  Y}_d(u),{\mathcal Y}_d'(u))J_\phi(u)du,
$$ 
being $J_\phi(u)=(1-\|u\|^2)^{-1/2}$ the jacobian of $\phi$. Rescaling $u=v/\sqrt{d}$
$$N_{{\mathcal Y}_d}(\tilde R)
=\sum_{q=0}^\infty\int_{\sqrt d \tilde R}
G_q\left({\mathcal Y}_d\left(\frac v{\sqrt d}\right),{\mathcal
  Y}_d'\left(\frac v{\sqrt d}\right)\right)J_{\phi}\left(\frac{v}{\sqrt{d}}\right)dv.
$$
Besides, Rice formula, the domination for ${\cal H}_{q,d}$ given in \eqref{eq:cotaH}, 
the convergence of ${\mathcal Y}_d$ to ${\mathcal Y}_\infty$ in Proposition \ref{p:local} 
and the convergence, after normalization, of $\bar{R}$ to $[-1/2,1/2]^m$ in Proposition \ref{p:partition} 
yield
\begin{equation}\label{eq:conv-tan}
\Var(N_{{\mathcal Y}_d}(\tilde R))\mathop{\to}_{d\to\infty}
  \Var\left(N_{{\mathcal Y}_\infty}\left(\left[-\frac12,\frac12\right]^m\right)\right). 
\end{equation}
In fact, 
writing $\Var(N)=\E(N(N-1))-(\E(N))^2+\E(N)$, 
for the first term we have 
\begin{multline*}
\E[N_{\bY}(R)(N_{\bY}(R)-1)]\\
=d^m\int_{\tilde R\times\tilde R}\E[|\det {\mathcal Y}'_d(u)||\det
  {\mathcal Y}'_d(v)|\,|{\mathcal Y}_d(u)={\mathcal
  Y}_d(v)=0]p_{u,v}(0,0)J_{\phi}(u)J_{\phi}(v)dudv\\
=\int_{\sqrt d\tilde R\times\sqrt d\tilde R}
\E\left[\left|\det {\mathcal Y}'_d\left(\frac u{\sqrt d}\right)\right|\left|\det {\mathcal Y}'_d\left(\frac v{\sqrt d}\right)\right|\mid
{\mathcal Y}_d\left(\frac u{\sqrt d}\right)={\mathcal Y}_d\left(\frac v{\sqrt d}\right)=0\right]\\
\cdot
  p_{\frac{u}{\sqrt{d}},\frac{v}{\sqrt{d}}}(0,0)J_{\phi}\left(\frac{u}{\sqrt{d}}\right)J_{\phi}\left(\frac{v}{\sqrt{d}}\right)dudv.
  \end{multline*}  
Then,
\begin{align*}
    &\lim_{d\to\infty}
  \E[N_{\bY}(R)(N_{\bY}(R)-1)]=\\
 &\quad =\int_{\left[\frac12,\frac12\right]^m\times\left[\frac12,\frac12\right]^m}
\E[|\det \mathcal Y'_{\infty}(u)||\det \mathcal Y'_{\infty}(v)|\mid\mathcal Y_{\infty}(u)=\mathcal Y_{\infty}(v)=0]\\
    &\qquad\qquad\qquad\qquad\qquad\cdot p_{{\cal Y}_\infty(u),{\cal Y}_\infty(v)}(0,0)dudv\\
    &\quad=\E\left[N_{{\mathcal Y}_{\infty}}\left(\left[\frac12,\frac12\right]^m\right)\left(N_{{\mathcal Y}_{\infty}}\left(\left[\frac12,\frac12\right]^m\right)-1\right)\right]
<\infty.
\end{align*}
The remaining terms are easier. 

The same arguments show that for all $q$ we have 
$$ V_{q,d}:=\Var(\pi_q(N_{\bY}(R)))\mathop{\to}_{d\to\infty}
  \Var\left(\pi_q\left(N_{{\cal
  Y}_\infty}\left(\left[-\frac12,\frac12\right]^m\right)\right)\right)=:
  V_{q,\infty}.
  $$
Thus, for all $Q$ it follows that  
$
\sum^Q_{q=0}V_{q,d}
\mathop{\to}_{d\to\infty} \sum^Q_{q=0}
V_{q,\infty}.
$
By Parseval's identity, \eqref{eq:conv-tan} can be written as
$$
\lim_{d\to\infty}
\sum^{\infty}_{q=0}V_{q,d} =\sum^{\infty}_{q=0}V_{q,\infty}.
$$
Thus, by taking the difference we get 
\begin{equation}\label{eq:diff}
\lim_{d\to\infty}
\sum_{q> Q}V_{q,d}= \sum_{q>Q}V_{q,\infty}. 
\end{equation}
Given that the series $\sum^\infty_{q=0}V_{q,\infty}$ is convergent, 
we can choose $Q_0$ such that for $Q\geq Q_0$ it holds $\sum^\infty_{q>Q}V_{q,\infty}\le \varepsilon/2$. 
Hence, for this $Q_0$ and by using \eqref{eq:diff} we can choose $d_0$ such that for all $d>d_0$ and $Q\geq Q_0$
$$\sum_{q>Q}V_{q,d}\le \varepsilon.$$ 
Namely, there exits $d_0$ such that for $Q\geq Q_0$
$$
\sup_{d>d_0}\E[(\pi^Q(N_{\bY}(R))^2]<\varepsilon.
$$
The same can be written in the following form:  
there exists $D(Q)$ a sequence that tends to zero when $Q\to\infty$, such that
$$\sup_{d>d_0}\E[(\pi^Q(N_{\bY}(R))^2]<D(Q).$$
This concludes the proof.
\end{proof}

\section{Technical Proofs}
\subsection{The partition of the sphere}\label{s:partition}
In this section we describe a convenient essential partition of the unit sphere $S^m$ of $\R^{m+1}$. 
We use hyperspheric coordinates \eqref{e:hs} 
and two speeds 
\begin{equation}\label{eq:2speeds}
r = \frac1{\sqrt{d}} \textrm{ and } \bar{r} = r^\alpha,\; 0<\alpha < \frac{1}{m}.  
\end{equation}
We suppose that $r$ is sufficiently small so that
\begin{equation*} \label{e:1}
\sin\left(\frac{\bar{r}}{2}\right) \geq \frac{\bar{r}}{4},\quad 
r\leq\frac{\bar{r}}{2}.
\end{equation*}
{\bf Step 1} We begin with $\theta_1$. 
Set $r_1 = r$. 
Let $a$ be the minimal integer such that $\pi/2 -a r_1 \leq \bar{r}$. 
Thus, the segment $[\pi/2 -a r_1 ,\pi/2 +a r_1]$  is cut into $2a$  sub-intervals $I_{1,i_1}$, with centers $ \theta_{1,i_1}$, $i_1 = 1, \ldots 2a $.
Hence, 
\begin{equation} \label{e:2}
\{\theta_{1,i_1}: i = 1, \ldots, 2a \} 
\subset\left[ \frac{\bar r}2 , \pi -\frac{\bar r}2\right].
\end{equation}
{\bf Step 2}  Depending on the interval $I_{1,i_1}$ in which is located $\theta_1$ we set 
$$r_{2,i_1}  = \frac{r} {\sin( \theta_{1,i_1})}.$$
Note that because of \eqref{e:2} uniformly over all possible values of $ \theta_{1,i_1}$
$$
\sin( \theta_{1,i_1})\geq \frac{\bar{r}}{4},
$$
implying that 
$$
r_{2,i_1} \leq 4 r ^{1-\alpha}.
$$
We impose again $r$ to be sufficiently small such that $r_{2,i_1} \leq \bar r /2$, 
this is possible by the last inequality and \eqref{eq:2speeds}. 
   
We then decompose 
the interval of variation of $\theta_2$, $[0,\pi)$, in the same manner using $r_{2,i_1}$ instead of $r_1$.
The intervals are now denoted by $I_{2,i_1,i_2}$, their centers by $\theta_{2,i_1,i_2}$ and their number by $a_{2,i_1}$.
   
\noindent{\bf Step j}
For $\theta_1 \in I_{1,i_1}$, $\theta_2 \in I_{2,i_1, i_2}$, \ldots, $\theta_{j-1}\in I_{j-1,i_1, i_2, \ldots,i_{j-1}} $ we  set
$$r_{j,i_1,\dots,i_{j-1}} 
=\frac{r_{j-1,i_1,\dots,i_{j-2}}}{\sin(\theta_{j-1, i_1, i_2, \ldots,i_{j-1} })}
=\frac{r} {\sin( \theta_{1,i_1})  ... 
\sin(\theta_{j-1, i_1, i_2, \ldots,i_{j-1} })}.$$
  
By construction, all the  sinus in the denominator  are greater than $\bar r /4$  implying that 
$$ 
r_{j,i_1,\dots,i_{j-1}} \leq  4^{j-1}   r ^{1- (j-1)\alpha}.
$$
We impose again $r $ to be sufficiently small  such that $r_{j,i_1,\dots,i_{j-1}} \leq \bar r /2$. 
    
Then we cut again the interval $[0,\pi)$ into sub-intervals in the same manner.  
The intervals are now  denoted by  $I_{j ,i_1, i_2, \ldots,i_{j}}$, their centers by $ \theta_{j ,i_1, i_2, \ldots,i_{j}}$ and their number by $a_{j,i_1, \ldots i_{j-1}}$.
 
\noindent{\bf Step m} The last step differs on two points: first we divide  the interval $[0,2\pi]$, and second  
we divide it  entirely  except rounding problems.  \\

The exceptional set $E$ of the sphere not covered by the sets above 
is included in the set 
$$ 
\bigcup^{m}_{i=1}x^{(m)}\left(\left\{\theta_i \in  [0, \bar r] \cup [\pi -\bar r , \pi]\right\}\right).
$$
Excluding $E$, we have made an essential partition of the sphere in hyperspherical rectangles (HSR) of the type 
$$
 R(i_1,\ldots,i_m)= \{\theta_1 \in I_{1,i_1},   \theta_2   \in I_{2,i_1, i_2} ,  \ldots  \theta_{m}   \in I_{m ,i_1, i_2, \ldots,i_{m}}\} 
$$

We are in position to prove Proposition \ref{p:partition}.
\begin{proof}[Proof of Proposition \ref{p:partition}]

1. Since $E\subset 
\bigcup^{m}_{i=1}x^{(m)}\left(\left\{\theta_i \in  [0, \bar r] \cup [\pi -\bar r , \pi]\right\}\right)$,   its $m$-dimensional volume $\Vol(E)=O(\bar{r})$ tends to zero. 

Using Lemma \ref{l:jm}  $$\Var(N_{\bY}(E))=o(d^{m/2}).$$\\
      
2. Recall that we are using the geodesical distance \eqref{eq:geodist}. 
Let $|(i_1,\ldots,i_m)|=i_1+\ldots+i_m$. 
We want to prove that 
If $|(i_1,\ldots,i_m)-(i'_1,\ldots,i'_m)|\geq 2$ then 
$\dist(R(i_1,\ldots,i_m),R(i'_1,\ldots,i'_m))\geq r$. 

Let us compute the inner product for $\theta'=\theta+\gamma \bar{r}_k e_k$.
\begin{align*}
\< x^{(m)}(\theta),x^{(m)}(\theta') \> 
&= 1+\< x^{(m)}(\theta),x^{(m)}(\theta') -x^{(m)}(\theta)\> \\
&= 1+\frac{\gamma^2}{2}\bar{r}^2_k\< x^{(m)}(\theta),\partial^2_kx^{(m)}(\theta)\> +{O}(\bar{r}^2_k)\\
&= 1-\frac{\gamma^2}{2}\bar{r}^2_k\prod^{k-1}_{j=1}\|x^{(m-k)}(\theta_{k+1},\ldots,\theta_m)\|^2+{O}(\bar{r}^2_k)\\
&= 1-\frac{\gamma^2}{2}r^2+{O}(\bar{r}^2_k).
\end{align*}
The result follows.\\

3. 
Consider a fixed HSR $R(i_1,\dots,i_m)$ with center   
$\big(  \theta_{1 ,i_1},\dots  ,\theta_{m ,i_1, i_2, \ldots,i_{m}}\big)$ and side lenghts $r_1,\ldots,r_{m,i_1,\ldots,i_m}$. 
For short we write $\bar{\theta}=( \bar \theta_1, \ldots , \bar \theta_m)$ for the center and $\bar{r}_1,\ldots,\bar{r}_m$ for the side lengths. 

The generic coordinates of a point on the HSR are
\begin{equation*}
\theta=\big(  \bar \theta_1 +u_1\bar{r}_1, \ldots , \bar \theta_m + u_m \bar{r}_m \big); 
\quad  {\mathbf u}=(u_1,\ldots,u_m) \in \left[-\frac12,\frac12\right]^m.
\end{equation*}
The corresponding cartesian coordinates $x^{(m)}(\theta)$, $x^{(m)}(\bar{\theta})$ are given by \eqref{e:hs}. 
The tangent space is computed by differentiating  in \eqref{e:hs} with respect  to $u_1, \ldots, u_m$. 
An orthonormal basis is given by  
\begin{equation*}
T_k=\left(\begin{array}{c}
0_{k-1}\\ -\sin(\bar{\theta}_k)\\
\cos(\bar{\theta}_k)x^{(m-k)}(\bar{\theta}_{k+1},\ldots,\bar{\theta}_m)
\end{array}\right);\;k=1,\ldots,m.
\end{equation*}
Here, $0_{k-1}$ stands for a vector of $k-1$ zeros 
(this is ommited when $k=1$). 
The projection $\phi$ can be computed easily on this basis 
taking inner products. 
Let us perform one of these inner products, the rest are similar. 
$$
\< x^{(m)}(\theta)-x^{(m)}(\bar{\theta}),T_k \>
=-\sin(\bar{\theta}_k)\Delta_k
+\cos(\bar{\theta}_k)
\< x^{(m-k)}(\bar{\theta}_{k+1},\ldots,\bar{\theta}_m), \bar{\Delta}_k \>,
$$
with
$$
\Delta_k=\prod^{k-1}_{j=1}\sin(\theta_j)\cdot\cos(\theta_k)
-\prod^{k-1}_{j=1}\sin(\bar{\theta}_j)\cdot\cos(\bar{\theta}_k)
$$
and
$$
\bar{\Delta}_k=
\prod^{k}_{j=1}\sin(\theta_j)\cdot x^{(m-k)}(\theta_{k+1},\ldots,\theta_m)
-\prod^{k}_{j=1}\sin(\bar{\theta}_j)\cdot x^{(m-k)}(\bar{\theta}_{k+1},\ldots,\bar{\theta}_m).
$$
Since the sine and the cosine functions have bounded  second derivative  for $ i=1,\ldots,m$
\begin{align*}
\sin(\theta_i) & = \sin(\bar  \theta_i)  +   \bar{r}_i  u_i  \cos(\bar  \theta_i)  +  O( \bar{r}_i^2), \\
\cos(\theta_i) & = \cos(\bar  \theta_i)   - \bar{r}_i  u_i  \sin(\bar  \theta_i)  +  O( \bar{r}_i^2).
\end{align*}
The implicit constants in the $O$ notation do not depend on the indexes $i_j$. 
Hence, the dominant terms in the differences $\Delta_k$ and $\bar{\Delta}_k$ are those where only one of the $\theta_j$ differs from $\bar{\theta}_j$, 
the rest are of higher order as $r\to0$.

Recall that the construction has been performed in such a way that the  quantities $\bar{r}_1,\ldots, \bar{r}_m$ tend to zero  uniformly as $r$ tends to $0$.
Let us study first the case of $\theta_j=\bar{\theta}_j$ for $j\neq k$, in this case
\begin{multline*}
\< x^{(m)}(\theta)-x^{(m)}(\bar{\theta}),T_k \>
=-\sin(\bar{\theta}_k)\prod^{k-1}_{j=1}\sin(\bar{\theta}_j)\cdot(-\bar{r}_ku_k\sin(\bar{\theta}_k))\\
+\cos(\bar{\theta}_k)\prod^{k-1}_{j=1}\sin(\bar{\theta}_j)\cdot(\bar{r}_ku_k\cos(\bar{\theta}_k))\\
\cdot\< x^{(m-k)}(\bar{\theta}_{k+1},\ldots,\bar{\theta}_m), x^{(m-k)}(\bar{\theta}_{k+1},\ldots,\bar{\theta}_m)\>+  O(r^2)\\
=\prod^{k-1}_{j=1}\sin(\bar{\theta}_j)\cdot \bar{r}_k\cdot u_k+  O( r^2)
=ru_k+  O( r^2).
\end{multline*}
Here we used that  by construction $\prod^{k-1}_{j=1}\sin(\bar{\theta}_j)\cdot \bar{r}_k=r$. 
By the same arguments one can show that in the case $\theta_k=\bar{\theta}_k$ the terms of the difference 
are uniformly $o(r)$. 
Hence, uniformly
$$
\frac1r\< x^{(m)}(\theta)-x^{(m)}(\bar{\theta}),T_k \> \mathop{\to}_{r\to 0} u_k.
$$
Consequently,  since we have a finite number of coordinates, the result on the convergence in the Hausdorff metric follows.
\end{proof}

\subsection{Chaotic expansions and contractions}\label{s:kernels}
In this section we write the chaotic components $I_{q,d}$ 
in Proposition \ref{prop:expansion} 
as multiple stochastic integrals wrt a standard Brownian motion $\bB$ 
and use this fact in order to 
prove Lemma \ref{lemma:cont-cov}.

Let $\bB=\{B(\lambda):\lambda\in[0,\infty)\}$ be a standard Brownian motion on $[0,\infty)$. 
By the isometric property of stochastic integrals 
there exist kernels $h_{t,\ell}$ such that 
the components of the vector $\bhZ$ defined in \eqref{eq:zeta} can be written as:
\begin{equation}\label{eq:h}
Z_\ell(t)=\int^\infty_0h_{t,\ell}(\lambda)dB(\lambda),\,\ell=1,\dots,m(m+1).
\end{equation}
The kernels $h_{t,\ell}$ can be computed explicitly 
from the definition of $Z_\ell$ 
writting the random coefficients as integrals wrt the Brownian motion. 

%
We quickly recall some properties of contractions \eqref{eq:cont} and multiple stochastic integrals, 
see \cite{pta} for details. 
Note that for $f,g\in L^2([0,\infty)^q)$, $f\otimes_0g=f\otimes g$ is the tensorial product and 
$f\otimes_qg=\<f,g\>$ is the inner product in $L^2_s([0,\infty)]^q)$. 
Besides, if $f=\bar{f}^{\otimes q}$ and $g=\bar{g}^{\otimes q}$, then, 
$f\otimes_ng=\<\bar{f},\bar{g}\>^n\bar{f}^{\otimes q-n}\otimes\bar{g}^{\otimes q-n}$ 
where this time the inner product is in $L^2([0,\infty))$.

Let $h\in L^2([0,\infty))$ and $I^{\bB}_1(h)=\int^\infty_0h(\lambda)dB(\lambda)$, 
then $H_q(I^\bB_1(h))=I^{\bB}_q(h^{\otimes q})$ where $I^{\bB}_q$ 
is the $q$-folded multiple stochastic integral wrt $\bB$ 
and $H_q$ the $q$-th Hermite polynomial. 
A key property of stochastic integrals is multiplication formula, 
for $f\in L^2_s([0,\infty)^p)$, $g\in L^2_s([0,\infty)^q)$,
\begin{equation*}
I^{\bB}_p(f)I^{\bB}_q(g)=\sum^{\min\{p,q\}}_{n=0}n!\binom{p}{n}\binom{q}{n}
I^{\bB}_{p+q-2n}(f\otimes_n g).
\end{equation*}
Note that if $f=\bar{f}^{\otimes p}$, $g=\bar{g}^{\otimes q}$ 
and $\bar{f},\bar{g}$ are orthogonal in $L^2([0,\infty))$, 
then, $I^{\bB}_p(f)I^{\bB}_q(g)=I^{\bB}_{p+q}(f\otimes g)$. 
Finally, let us mention that 
if $f\in L^2([0,\infty)^q)$ then 
$I^{\bB}_q(f)=I^{\bB}_q(\tilde{f})$ 
being $\tilde{f}$ the symmetrization of $f$, that is,
$$
\tilde{f}(\lambda_1,\dots,\lambda_q)
=\frac{1}{q!}\sum_{\pi\in \PP_q}f(\lambda_{\pi(1)},\dots,\lambda_{\pi(q)}),
$$
being $\PP_q$ the symmetric group of order $q$.

Next Lemma expresses $I_{q,d}$ as a multiple stochastic integral wrt $\bB$. 
\begin{lemma}\label{lemma:kernels}
With the notations and assumptions of Proposition
  \ref{prop:expansion}, then, $I_{q,d}$ can be written as a multiple stochastic integral
\begin{equation*}
I_{q,d}=I^{\bB}_q(g_{q,d})=\int_{[0,\infty)^q}g_{q,d}({\bf\lambda})dB(\bf{\lambda});
\end{equation*}
with 
\begin{equation*}
g_{q,d}(\bf{\lambda})=
\frac{d^{m/4}}{2}\sum_{|\bgamma|=q}c_{\bgamma}
\int_{S^{m}}
(\otimes^{m(m+1)}_{\ell=1}h^{\otimes \gamma_\ell}_{t,\ell})(\bf{\lambda})
dt,
\end{equation*}
where $h_{t,\ell}$ is defined in \eqref{eq:h}.  
\end{lemma}

\begin{proof}
In the first place, we plug in the expressions for $Z_\ell$ 
given in \eqref{eq:h} 
in the formula for $I_{q,d}$ 
given in Proposition \ref{prop:expansion}. 
Namely,
$$
I_{q,d}=\frac{d^{m/4}}{2}\sum_{|\bgamma|=q}c_{\bgamma}
\int_{S^m}
\prod^{m+m^2}_{\ell=1}
H_{\gamma_\ell}(I^{\bB}_1(h_{t,\ell}))dt
$$
Now, 
using the relation between Hermite polynomials and stochastic integrals we get  
\begin{align*}
I_{q,d}&=\frac{d^{m/4}}{2}\sum_{|\bgamma|=q}c_{\bgamma}\int_{S^m}
\prod^{m+m^2}_{\ell=1}
I^{\bB}_{\gamma_\ell}(h^{\otimes \gamma_\ell}_{t,\ell})
dt\\
&=\frac{d^{m/4}}{2}\sum_{|\bgamma|=q}c_{\bgamma}
\int_{S^m}
I^{\bB}_{\bgamma|}
(\otimes^{m+m^2}_{\ell=1}h^{\otimes \gamma_\ell}_{t,\ell})dt.
\end{align*}
Last equality is consequence of 
the multiplication formula for stochastic integrals and the orthogonality of the kernels $h_{t,\ell}$, see Section \ref{s:anciliary}.  
Finally, by the Fubini stochastic Theorem. 
$$
I_{q,d}=I^{\bB}_{|\bgamma|}
\left(\frac{d^{m/4}}{2}\sum_{|\bgamma|=q}c_{\bgamma}\int_{S^m}
\otimes^{m+m^2}_{\ell=1}h^{\otimes \gamma_\ell}_{t,\ell}
dt\right)
=I^{\bB}_{q}(g_{q,d}).
$$
This concludes the proof.
\end{proof}

\begin{proof}[Proof of Lemma \ref{lemma:cont-cov}]
For simplicity let us write 
$\tilde{g}_{q,d}(\bf{\lambda})
=d^{m/4}\int_{S^m}G_{t,q}(\bf{\lambda})dt$ with
\begin{equation*}
G_{t,q}(\bf{\lambda})
=\frac{1}{q!}\sum_{\pi\in \PP_q}\sum_{|\bgamma|=q}c_{\bgamma}
(\otimes^{m(m+1)}_{\ell=1}h^{\otimes \gamma_\ell}_{t,\ell})(\bf{\lambda}_\pi),
\end{equation*}
being $\bf{\lambda}_\pi=(\lambda_{\pi(1)},\dots,\lambda_{\pi(q)})$. 
Note that
$$
\tilde{g}_{q,d}\otimes_n \tilde{g}_{q,d}(\bf{\lambda}_{(q-n)})
=d^{m/2}\int_{S^m\times S^m}[G_{s,q}\otimes_n G_{t,q}](\bf{\lambda}_{(q-n)})dsdt.
$$
The subscript of the vector $\bf{\lambda}$ stands for its dimension. 
Besides. since $G_{t,q}$ is the tensorial product 
of kernels $h_{t,\ell}$ in $L^2_s([0,\infty))$, 
the last contraction can be expressed in terms of 
the contractions of these basic kernels $h_{t,\ell}$. 
Besides, according to the isometric property 
of stochastic integrals we have 
\begin{multline*}
[h_{s,k}\otimes_1h_{t,\ell}](\lambda)
=\int^\infty_0 h_{s,k}(\lambda)h_{t,\ell}(\lambda)d\lambda\\
=\E\left[\int^\infty_0 h_{s,k}(\lambda)dB(\lambda)\cdot\int^\infty_0 h_{t,\ell}(\lambda)dB(\lambda)\right]
=\E(Z_k(s)Z_\ell(t))\\
=\rho_{k\ell}(s,t),
\end{multline*}
as defined in \ref{eq:covZ}. 
Note that when the covariances $\rho_{k\ell}(s,t)$ do not vanish, 
they are either $r_d(\<s,t\>)$ or its spherical derivatives 
$\partial_{t_k}r_d(\<s,t\>)=r'_d(\<s,t\>)\partial_{t_k}(\<s,t\>),$
where $\partial_{t_k}$ means that we fix $s$ and take derivative to the $k$-th direction, etc.
Note also that $|\partial_{t_k}\<s,t\>|\leq1$.

Hence, we can write
\begin{equation*}
[G_{s,q}\otimes_n G_{t,q}](\bf{\lambda}_{(q-n)})
=\frac{1}{(q!)^2}\sum_{\pi,\pi^\prime\in \PP_q}
\sum_{|\bgamma|=|\bgamma'|=q}c_{\bgamma} c_{\bgamma'} 
\rho^{(n)}_{q,d}(s,t)\bar{G}_{2q-2n}(\bf{\lambda}_{\pi,\pi'});
\end{equation*}
where $\rho^{(n)}_{q,d}(s,t)$ 
is a product of covariances of $\bhZ$ 
with total degree $n$ while $\bar{G}_{2q-2n}(\lambda_{\pi,\pi'})$ is a tensor product of kernels $h_{t,\ell}$  
with degree $2q-2n$ and the coordinates are permuted according to $\pi$ and $\pi'$. 
Note that which covariances are involved in the $\rho$'s 
depend on the indexes $\bgamma,\bgamma'$.

Therefore, writting ${\boldsymbol\pi}=(\pi_1,\pi_2,\pi_3,\pi_4)$,
\begin{multline*}
\|g_{q,d}\otimes_n g_{q,d}\|^2_{2}\leq d^m
\frac{1}{(q!)^2}\sum_{{\boldsymbol\pi}\in (\PP_q)^4}
\sum_{|\bgamma|=|\bgamma|=q}
c_{\bgamma}  c_{\bgamma'}\\
\cdot\int_{(S^m)^4}
\rho^{(n)}_{q,d}(s,t)
\rho^{(n)}_{q,d}(s',t')
\cdot\rho^{(q-n)}_{q,d}(t,t')
\rho^{(q-n)}_{q,d}(s,s')dsdtds'dt'.
\end{multline*}
The variance of $Z_\ell$ restricted to the sphere is constantly $1$, 
Cauchy-Schwarz inequality implies that 
we can bound the absolute value of the power of any covariance by the very covariance. 
Hence, we can bound each term in the last sum by  a term of the form 
(up to a constant)
\begin{equation*}
d^m
\int_{(S^m)^4}|r^{(k_1)}_{d}(s,t)
r^{(k_2)}_{d}(s',t')
r^{(k_3)}_{d}(t,t')
r^{(k_4)}_{d}(s,s')|dsdtds'dt',
\end{equation*}
where $k_j=0,1,2$ indicates a derivative of order $0,1$ or $2$ of $r_d$. 
Since each covariance is a function of the inner product of its arguments, 
they are invariant under isometries. 
Thus, consider isometries $U_s$ such that $U_s(s)=e_0$ 
and $U_{t'}$ such that $U_{t'}(t')=e_0$. 
Then, \eqref{eq:cont-cov} can be written as
\begin{multline*}
d^m
\int_{(S^m)^4}
|r^{(k_1)}_{d}(\<e_0,U_s(t)\>)
r^{(k_2)}_{d}(\<U_{t'}(s'),e_0\>)|\\
\cdot|r^{(k_3)}_{d}(\<U_{t'}(t),e_0\>)
r^{(k_4)}_{d}(\<e_0,U_s(s')\>|dsdtds'dt'.
\end{multline*}
Introducing the, isometric, change of variables 
$\tau_1=U_s(t)$, $\tau_2=U_{t'}(s')$, 
$\tau_3=U_{t'}(t)$ and $\tau_4=s'$ and bounding 
$|r^{(k_4)}_{d}(\<e_0,U_s(s')\>)|\leq 1$ 
we get that the last expression is less or equal than
\begin{multline*}
d^m
\int_{(S^m)^4}
|r^{(k_1)}_{d}(\<e_0,\tau_1\>)
r^{(k_2)}_{d}(\<e_0,\tau_2\>)
r^{(k_3)}_{d}(\<e_0,\tau_3\>)|d\tau_1d\tau_2d\tau_3d\tau_4\\
=C_{m}d^{m}
\prod^3_{j=1}\int_{S^m}
|r^{(k_j)}_{d}(\<e_0,\tau_j\>)|d\tau_j
=C_{m}d^{m}
\prod^3_{j=1}\int_0^\pi
\sin^{m-1}(\theta)
|r^{(k_j)}_{d}(\cos(\theta))|
d\theta\\
=C_{m}d^{m}
\prod^3_{j=1}\int^{\pi/2}_{0}
\sin^{m-1}(\theta)
|r^{(k_j)}_{d}(\cos(\theta))|
d\theta,
\end{multline*}
where we used and the symmetry wrt $\theta=\pi/2$ of the integrand. 

The result follows. 
\end{proof}

\subsection{Anciliary computations}\label{s:anciliary}
We start the computations with the covariances of the vector $(\bhZ(s),\bhZ(t))$ defined in \eqref{eq:zeta}, see \cite{aadl}. 
Actually, by the definition of KSS distribution, it suffices to consider
$$
\left(Y_\ell(s),Y_\ell(t),\overline{{\bold Y}}'_\ell(s),\overline{{\bold Y}}'_\ell(t)\right).
$$
for a fixed $\ell=1,\ldots,m$. 
Its variance-covariance matrix 
can be writen in the following form
\begin{eqnarray*}
\left[\begin{array}{c|c|c}
A_{11}&A_{12} &A_{13}\\ \hline
A_{12}^\top&I_m\,\,&\,A_{23}\\ \hline
A_{13}^\top&A_{23}^\top\,\,&I_m\\
\end{array}\right],
\end{eqnarray*}
where $I_m$ is the $m\times m$ identity matrix, 
$$
A_{11}=\left[\begin{array}{cc}
1&\mathcal{C}\\
\mathcal{C}&1
\end{array}\right], \;
A_{12}= \left[\begin{array}{cccc}
0& 0&\cdots& 0\\
-\mathcal{A}&0& \cdots & 0
\end{array}\right],\:
A_{13}=  \left[\begin{array}{cccc}
\mathcal{A}& 0&\cdots& 0\\
0&0& \cdots & 0
\end{array}\right],$$
and $A_{23}=\diag( [\mathcal{B},\mathcal{D},\ldots,\mathcal{D}])_{m\times m}$. 
The quantities ${\cal A}$, ${\cal B}$, ${\cal C}$ and ${\cal D}$ were defined in \eqref{eq:deriv}.

Gaussian regression formulas, see \cite{aw}, imply that the conditional distribution of the vector 
$(\overline{Y}'_\ell(s),\overline{Y}'_\ell(t))$, 
conditioned on $\bY(s)=\bY(t)=0$, 
is centered normal with variance-covariance matrix given by
\begin{equation*}  
\left[\begin{array}{c|c}
B_{11}&B_{12} \\ \hline
B_{12}^\top&B_{22} \\
\end{array}\right],
\end{equation*}
with $B_{11}=B_{22}= \diag( [\sigma^2, 1,\dots,1])$ and 
$B_{12}= \diag( [\sigma^2 \rho, \mathcal{D},\dots,\mathcal{D}])$. \\

Now we move to the proof of some lemmas.
\begin{lemma}\label{l:Hpar}
The function ${\cal H}_{q,d}$ defined in \eqref{e:H} is even.
\end{lemma}
\begin{proof}
We need to make explicit the multi-indexes. 
\begin{multline}\label{e:mehler}
\E[\bH_{\balpha}(\mathbf Y(s))\bH_{\bbeta}(\overline{\mathbf Y}'(s))\bH_{\balpha'}(\mathbf Y(t))\bH_{\bbeta'}(\overline{\mathbf Y}'(t))]\\
=\prod_{\ell=1}^m
\E\left[H_{\alpha_\ell}(Y_{\ell}(s))H_{\alpha'_\ell}(Y_\ell(t))
\cdot\prod_{k=1}^mH_{\beta_{\ell k}}(\overline{Y}'_{\ell k}(s))
\cdot\prod_{k'=1}^mH_{\beta'_{\ell k'}}(\overline{Y}'_{\ell k'}(t))
\right]\\
=\prod_{\ell=1}^m
\E[H_{\alpha_\ell}(Y_\ell(s))H_{\alpha'_\ell}(Y_\ell(t))
H_{\beta_{\ell 1}}(\overline{Y}'_{\ell 1}(s))H_{\beta'_{\ell 1}}(\overline{Y}'_{\ell 1}(t))]\\
\cdot\prod^m_{j=2}\E[H_{\beta_{\ell j}}(\overline{Y}'_{\ell j}(s))H_{\beta'_{\ell j}}(\overline{Y}'_{\ell j}(t))]. 
\end{multline}
In the second equality we use that the random vectors
$$
(Y_\ell(s),Y_\ell(t),\overline{Y}'_{\ell1}(s),\overline{Y}'_{\ell1}(t));\quad
(\overline{Y}'_{\ell j}(s),\overline{Y}'_{\ell j}(t));\quad j\geq2
$$ 
are independent. 

Using Mehler's formula, 
we get
$$
\E[H_{\beta_{\ell j}}(\overline{Y}'_{\ell j}(s))H_{\beta'_{\ell j}}(\overline{Y}'_{\ell j}(t))]
=\delta_{\beta_{\ell j}\beta'_{\ell j}} \beta_{\ell j}!\,(\rho''_{\ell j})^{\beta_{\ell j}},
$$
where $\rho''_{\ell j}=\rho''_{\ell j}(\<s,t\>)=\E(\overline{Y}'_{\ell j}(s)\overline{Y}'_{\ell j}(t))=\<t,s\>^{d-1}.$ 
Since $\sum_{j=1}^m\beta_{\ell j}$ is even, we have that  either $\beta_{\ell1}$ is even and then $\sum_{j=2}^m\beta_{\ell j}$ 
is even too or $\beta_{\ell1}$ is odd and in this case  $\sum_{j=2}^m\beta_{\ell j}$ is also odd. 

For the first factor, 
using again Mehler's formula we get 
$$
\E[H_{\alpha_\ell}(Y_\ell(s))H_{\alpha'_\ell}(Y_\ell(t))
H_{\beta_{\ell 1}}(\overline{Y}'_{\ell 1}(s))H_{\beta'_{\ell 1}}(\overline{Y}'_{\ell 1}(t))]=0,
$$
if $\alpha_\ell+\beta_{\ell 1}\neq \alpha'_{\ell}+\beta'_{\ell 1}$.
Otherwise, consider $\Lambda\subset \N^4$ defined by
$$
\Lambda=\{(d_1,d_2,d_3,d_4):d_1+d_2=\alpha_\ell,\,  
d_3+d_4=\beta_{\ell 1},\, d_1+d_3=\alpha'_\ell,\, d_2+d_4=\beta'_{\ell 1}\};
$$
then
\begin{multline*}
\E[H_{\alpha_\ell}(Y_\ell(s))H_{\alpha'_\ell}(Y_\ell(t))
H_{\beta_{\ell 1}}(\overline{Y}'_{\ell 1}(s))H_{\beta'_{\ell 1}}(\overline{Y}'_{\ell 1}(t))]=\\
\sum_{(d_i)\in\Lambda}\frac{\alpha_\ell ! \alpha'_\ell ! \beta_{\ell 1} ! \beta'_{\ell 1} !}
{d_1!d_2!d_3!d_4!}
\rho^{d_1}
(\rho')^{d_2}
(\rho')^{d_3}
(\rho'')^{d_4},
\end{multline*}
where $\rho=\rho(\<s,t\>)=\E(Y_\ell(s)Y_\ell(t))$, 
$\rho'=\E(Y_\ell(s)\overline{Y}'_{\ell 1}(t))$, 
$\rho'=\E(\overline{Y}'_{\ell 1}(s)Y_\ell(t))$
and 
$\rho''=\E(\overline{Y}'_{\ell 1}(s)\overline{Y}'_{\ell 1}(t))$.

Note that the conditions defining the index set $\Lambda$ implies that the first factor in Equation \eqref{e:mehler} is
$$
\prod^r_{\ell=1}
\sum_{(d_i)\in\Lambda}\frac{\alpha_\ell ! \alpha'_\ell ! \beta_{\ell 1} ! \beta'_{\ell 1} !}
{d_1!d_2!d_3!d_4!}
\rho^{d_1}
(\rho')^{d_2}
(\rho')^{d_3}
(\rho'')^{d_4},
$$
Hence, if we change $\<s,t\>$ by $-\<s,t\>$ in this expression, 
for each $\ell$ 
we have the factor 
\begin{eqnarray*}
(-1)^{dd_1}\cdot(-1)^{(d-1)(d_2+d_3)}\cdot(-1)^{dd_4}
=(-1)^{d(d_1+d_4)+(d-1)(d_2+d_3)}\\
=(-1)^{d\alpha_\ell}(-1)^{d\beta'_{\ell_1}}(-1)^{2(\alpha'_\ell-d_1)}=(-1)^{d\beta'_{\ell_1}}
\end{eqnarray*}

In summary, changing $\<t,s\>$ by $-\<t,s\>$ in (\ref{e:mehler}) and  considering each term for $j=1\ldots,m$ of the product,
either $\beta'_{\ell_1}$ and $\sum_{j=2}^m\beta'_{\ell_j}$ are even then the sign of this term does not change or  
the two numbers are odd and then they have a minus in front and the sign neither change. Thus we get that  the complete sign of (\ref{e:mehler}) does not change.
\end{proof}

\begin{proof}[Proof of Lemma \ref{l:normaG}]
By Hermite polynomials properties, we have
$$
\E[G_q^2(\zeta)]=
  \sum_{|\balpha|+|\bbeta|=q}b^2_{\balpha}\balpha!f^2_{\bbeta} \bbeta!,
$$
with $b_{\balpha}$ and $f_{\bbeta}$ defined in \eqref{eq:balpha} and \eqref{eq:fbeta} respectively. 
The following digression will be useful. 
Let  us consider two sequences $b_k$ and $a_k$ such that $\sum_{k=0}^\infty a_k^2<\infty$ and $b_k^2\to0$. 
We are interested in the sum
$$\sum_{k=0}^qa^2_{k}b^2_{q-k}=b^2_0a^2_q+b^2_1a^2_{q-1}+\ldots+b^2_{q}a^2_{0}.$$
By hypothesis $\sup_{k}b^2_k=||b^2||_{\infty}<\infty$, thus we get
\begin{eqnarray}
\label{cota}\sum_{k=0}^qb^2_{k}a^2_{q-k}\le||b^2||_{\infty}||a||_2^2.
\end{eqnarray}
Using this elementary fact we can explore the behavior of  our sum 
$$
d_{q}=\sum_{|\balpha|+|\bbeta|=q}b^2_{\balpha}f^2_{\bbeta}\balpha!\bbeta!=\sum_{|\balpha|+|\bbeta|=2l} b^2_{2\alpha_1}\ldots b^2_{2\alpha_m}f^2_{(\bbeta_1,\ldots,\bbeta_{m})}\balpha!\bbeta!.
$$ 
We affirm that  $b^2_{2\alpha}(2\alpha)!$ is decreasing, in fact
$$
\frac{b^2_{2(j+1)}(2(j+1))!}{b^2_{2j}(2j)!}=\left[\frac12\right]^2\frac{(2j+1)(2j+2)}{(j+1)^2}=\frac{(j+\frac12)}{j+1}<1.
$$
Moreover $b_0^2=\frac1{2\pi}<1$, then $||b||_{\infty}<1.$ 
Consider now
$$\E[G_q^2(\zeta)]=\sum_{k_{m+1}=0}^q\sum_{|\bbeta_2|+\ldots+|\bbeta_m|=k_{m+1}}\sum_{|\balpha|+|\bbeta_1|=q-k_{m+1}}b^2_{\balpha}f^2_{\bbeta}\balpha!\bbeta!.$$
Set $\alpha_i+\beta_{1 i}=l_i$, such that $q-k_{m+1}=l_1+\ldots+l_m$. 
In this form we obtain
$$
\sum_{|\balpha|+|\bbeta_1|=q-k_{m+1}}b^2_{\balpha}f^2_{\bbeta}\balpha!\bbeta!
=\prod_{i=1}^m\sum_{\alpha_i+\beta_{1i}=l_i}b_{\alpha_i}^2f^2_{\bbeta}\alpha_i!\bbeta_1!\prod_{j=2}^m\bbeta_{j}!.
$$
Using (\ref{cota}) it yields
$$\sum_{|\balpha|+|\bbeta_1|=q-k_{m+1}}b^2_{\balpha}f^2_{\bbeta}\balpha!\bbeta!\le \sum_{l_1+\ldots+l_m=q-k_{m+1}}\prod_{i=1}^m\sum_{\beta_{1i}=0}^{l_i}f^2_{\bbeta}\bbeta!.$$This allows us  getting the following bound
$$\E[G_q^2(\zeta)]\le ||f||^2_2.$$
The result follows.
\end{proof}

\bibliographystyle{imsart-nameyear}

\end{document}